\documentclass[12pt]{amsart}
\usepackage{amssymb}
\usepackage[latin1]{inputenc}
\usepackage{amsthm}
\usepackage{amsfonts}
\usepackage{mathrsfs}
\usepackage{graphicx}
\usepackage{amsmath}
\usepackage[all]{xy}

\let\mathcal\mathscr

\def\bC{{\mathbb C}}

\def\bR{{\mathbb R}}
\def\bP{{\mathbb P}}
\def\bN{{\mathbb N}}
\def\bQ{{\mathbb Q}}
\newtheorem{thm}{Theorem}[section]

\def\Q{{\bf Q}}

\def\codim{\mathop{\rm codim}\nolimits}

\def\NAmp{\mathop{\rm NAmp}\nolimits}

\def\Supp{\mathop{\rm Supp}\nolimits}

\def\tilde{\widetilde}

\def\phi{\varphi}

\def\cF{{\mathcal F}}
\def\cG{{\mathcal G}}

\numberwithin{equation}{section}

\newtheorem{theorem}[thm]{Theorem}
\newtheorem*{thma}{Theorem A}
\newtheorem*{thmb}{Theorem B}
\newtheorem*{thmc}{Theorem C}

\newtheorem*{thmf}{Theorem F}

\newtheorem{claim}[thm]{Claim}

\newtheorem{conjecture}[thm]{Conjecture}
\newtheorem{corollary}[thm]{Corollary}

\newtheorem*{cord}{Corollary D}
\newtheorem*{core}{Corollary E}

\newtheorem{definition}[thm]{Definition}
\newtheorem{example}[thm]{Example}

\newtheorem{lemma}[thm]{Lemma}

\newtheorem{proposition}[thm]{Proposition}
\newtheorem{prop}[thm]{Proposition}
\newtheorem{question}[thm]{Question}

\newtheorem{remark}[thm]{Remark}

\begin{document}

\title {Higher dimensional tautological inequalities and applications}
\author{Carlo Gasbarri, Gianluca Pacienza and Erwan Rousseau}
\date{}
\thanks{G. P. is partially supported by the project CLASS (ANR-2010-JCJC-0111-01) of the Agence Nationale de la Recherche.}
\thanks{E. R. is partially supported by project COMPLEXE (ANR-08-JCJC-0130-01) of the Agence Nationale de la Recherche.}
\keywords{Foliations; canonical singularities; Ahlfors currents; tautological inequalities; algebraic degeneracy; Green-Griffiths-Lang conjecture}
\subjclass[2000]{32H25, 32H30, 32S65.}

\begin{abstract}
We study the degeneracy of holomorphic mappings tangent to holomorphic foliations on projective manifolds. Using Ahlfors currents in higher dimension, we obtain several strong degeneracy statements.
\end{abstract}

\maketitle

\tableofcontents

\section{Introduction}
In the last decades, many efforts have been done to understand the geometry of subvarieties of varieties of general type. One of the main motivation is the fascinating conjectural relation between analytic aspects and arithmetic ones. On the geometric side, the philosophy (Green-Griffiths, Lang, Vojta, Campana) is that positivity properties of the canonical bundle of a projective manifold should impose strong restrictions on its subvarieties.

One of the first striking results is the following theorem of Bogomolov \cite{Bo} for surfaces.
\begin{theorem}[Bogomolov]
There are only finitely many rational and elliptic curves on a surface of general type with $c_1^2 > c_2$.
\end{theorem}
In this theorem, the hypothesis $c_1^2 > c_2$ ensures that the cotangent bundle is big, so that rational and elliptic curves are shown to be leaves of a foliation and then, one can use results on algebraic leaves of foliations \cite{Jou}.

Two decades later, this result was extended to transcendental leaves of foliations by McQuillan \cite{McQ0}.
\begin{thm}[McQuillan] Let $X$ be a surface of general type and $\cF$ a holomorphic foliation on $X$. Then $\cF$ has no entire leaf which is Zariski dense. 
\end{thm}
As a consequence he obtains the following. 
\begin{corollary}[McQuillan]
On a surface $X$ of general type with $c_1^2 > c_2$, there is no entire curve $f: \bC \to X$ which is Zariski dense.
\end{corollary}

It is of course of great interest to generalize these results, even partially, to higher dimension. On the algebraic side, this was investigated by Lu and Miyaoka in \cite{Lu-Mi}.
\begin{theorem}[Lu-Miyaoka]
Let $X$ be a nonsingular projective variety. If $X$ is of general type, then $X$ has only a finite number of nonsingular codimension-one subvarieties having pseudoeffective anticanonical divisor. In particular, $X$ has only a finite number of nonsingular codimension-one Fano, Abelian, and Calabi-Yau subvarieties.
\end{theorem}

This can be seen as a generalization to higher dimension of the aforementioned theorem of Bogomolov.

In this paper, we would like to study some generalizations of McQuillan's result. The study of several variables holomorphic maps into $X$ in relation with the hyperbolicity properties of $X$ has been considered in \cite{PaRou}, where the extension to this larger framework of the Demailly-Semple jet-spaces technology was given, together with several applications. 
Here, for $1\leq p\leq n-1$, we consider holomorphic mappings $f: \bC^{p} \to X$ of generic maximal rank into a projective manifold of dimension $n$, such that the image of $f$ is tangent to a holomorphic foliation $\mathcal{F}$ on $X$. We obtain several results of algebraic degeneracy in the strong sense (i.e. the existence of a proper closed subset of $X$ containing all such maps). In order to state our results we need to recall that if $L$ is a big line bundle on a projective variety $X$, the {\it non-ample locus} $\NAmp(L)$ is defined as follows 
$$
\NAmp(L):=\bigcap_{L\sim_\bQ A+E} \Supp(E)
$$
where the intersection runs over all the possible decompositions (up to $\bQ$-linear equivalence) of $L$ into the sum of an ample and an effective divisor. This locus (which is also called the {\it augmented base locus}) has the following property (cf. \cite{elmnp}):

\begin{equation}\label{eq:empty}
\NAmp(L)=\emptyset \Longleftrightarrow L\ \  {\textrm {is ample}} 
\end{equation}
The non-ample locus may be thought of as the locus outside which the line bundle $L$ is ample. 

If the foliation is smooth we obtain the following. 
\begin{thma}
Let $\mathcal{F}$ be a smooth foliation of dimension $p$ on a projective manifold $X$, with $p \leq \dim X=n$. If $X$ is of general type then the image of any holomorphic mapping $f:\bC^{p} \to X$ of generic maximal rank tangent to $\mathcal{F}$ is contained in $\NAmp(K_X)\subsetneq X$.
\end{thma}
Most of the examples of foliations are not smooth and the natural generalizations to the singular case may be highly non-trivial. For instance, in the 2-dimensional case treated by McQuillan, the technical core of the proof consists in the detailed study of the contribution of singularities of the foliations. 
On the other hand, in some cases, we have results reducing the study of non-smooth foliations to some special classes of singularities, see e.g. 
\cite{Sei, Can, McQPan}. In particular, thanks to the work of McQuillan, a class of singularities emerged as a very natural one: the class of canonical singularities (see Definition \ref{def:sing}). Notice that the more classical logarithmic simple singularities (see (\ref{eq:logsing})) fall in this class. 
In this framework we obtain several results. Supposing that the singularities of the foliation are of logarithmic simple type we show the following. 

\begin{thmb}
Let $\mathcal{F}$ be a holomorphic foliation of codimension one on a projective manifold $X$ of dimension $n$. Suppose $Sing \mathcal{F}$ consists only of logarithmic simple singularities. If the canonical line bundle $K_\cF$ of the foliation is big then the image of any holomorphic mapping $f: \bC^{n-1} \to X$ of generic maximal rank tangent to $\mathcal{F}$ is contained in $\NAmp(K_\cF)\subsetneq X$.
\end{thmb}

Theorem B may be seen as an illustration of the Green-Griffiths principle in the setting of holomorphic maps tangent to foliations for which we formulate the following generalized Green-Griffiths-Lang conjecture:
\begin{conjecture}[Generalized Green-Griffiths-Lang conjecture]
Let $(X, \mathcal{F})$ be a projective foliated manifold where $Sing \mathcal{F}$ consists only of canonical singularities and $K_\cF$ is big. Then there exists an algebraic subvariety $Y \subsetneq X$ such that any non-degenerate holomorphic map $f: \bC^p \to X$ tangent to $\mathcal{F}$ has image $f(\bC^p)$ contained in $Y$.
\end{conjecture}

In his recent work \cite{Dem10}, Demailly has also formulated a generalized Green-Griffiths-Lang conjecture but one should notice that the way he defines foliations of general type, using {\it admissible metrics}, is different. It seems an interesting question to compare the two definitions. One should insist here on the importance played by the singularities as stressed by the following example.
\begin{example}
Take a foliation $\mathcal{F}$ on $X=\bP^2$ with a Zariski dense entire curve $f: \bC \to (X, \mathcal{F})$ tangent to it. We have $K_\cF= \mathcal{O}(d_1-1)$ where $d_1$ is the degree of $\mathcal{F}$. Now, consider a birational map $g: \bP^2 \to \bP^2$ of degree $d_2$ and the foliation $\mathcal{G}:=g^*\mathcal{F}$. If $d_2$ is sufficiently large $K_\cG$ becomes positive. Nevertheless, the lifting of the entire curve shows that there exists a Zariski dense entire curve tangent to $\mathcal{G}$.
\end{example}

If the singularities are canonical and the foliation has local first integrals (see (\ref{eq:1int})) we get the following. 

\begin{thmc} 
Let $\mathcal{F}$ be a holomorphic foliation of codimension one on a projective manifold of general type $X$ of dimension $n$. Suppose $Sing \mathcal{F}$ consists only of canonical singularities with local first integrals. Then
the image of any holomorphic mapping $f:\bC^{n-1} \to X$ of generic maximal rank tangent to $\mathcal{F}$ is contained in $\NAmp(K_X)\subsetneq X$.
\end{thmc}

As a corollary of the classical result of Malgrange \cite{Mal} we deduce the following.
\begin{cord}
Let $\mathcal{F}$ be a holomorphic foliation of codimension one on a projective manifold $X$, $\dim X=n$. Suppose $X$ is of general type and $\codim Sing \mathcal{F} \geq 3$ where all singularities are canonical. Then
the image of any holomorphic mapping $f:\bC^{n-1} \to X$ of generic maximal rank tangent to $\mathcal{F}$ is contained in $\NAmp(K_X)\subsetneq X$.
\end{cord}

One should remark that, in the theorem above, we may suppose that $\codim Sing \mathcal{F} =2$ and that, locally there exists a formal first integral; in this case, by theorem 0.2 of \cite{Mal}, the formal integral will be convergent (holomorphic) and then apply Theorem C.

By a result due to Takayama \cite{Taka} on a variety $X$ of general type every irreducible component of $\NAmp(K_X)$ 
is uniruled. Therefore $\NAmp(K_X)$ is a very natural place for the images of the holomorphic maps to land in. Notice however that  in Theorems A and C we do not establish a direct link between the entire curves coming from the holomorphic maps and the rational curves given by Takayama's result. 
Notice also that when the relevant divisor ($K_X$ in Theorems A and C and $K_\cF$ in Theorem B) is ample, by (\ref{eq:empty}) the results above {\it exclude} the existence of maximal rank holomorphic mappings tangent to the foliations appearing in their statements.  An explicit and interesting illustration of such a situation is the following.

\begin{core}
\begin{enumerate}

\item Let $X_d\subset \bP^{n+1}$ a smooth hypersurface of degree $d>n+2$. 
Let $\mathcal{F}$ be a holomorphic foliation of codimension one on  $X_d$ such that $Sing \mathcal{F}$ consists only of canonical singularities with local first integrals. Then there is no
 holomorphic mapping $f:\bC^{n-1} \to X_d$ of generic maximal rank tangent to $\mathcal{F}$.
\item Let $\mathcal{F}_d$ be a holomorphic foliation of codimension one on $\bP^n$ of degree $d\geq n$. Suppose $Sing \mathcal{F}_d$ consists only of logarithmic simple singularities. Then there is no holomorphic mapping $f: \bC^{n-1} \to \bP^n$ of generic maximal rank tangent to $\mathcal{F}$. \end{enumerate}

\end{core}

We construct in \S 6 explicit examples of varieties of general type with foliations satisfying the hypotheses of the preceding theorems.

In dimension $3$ we obtain the following result. 
\begin{thmf}
Let $\mathcal{F}$ be a holomorphic foliation of codimension one on a projective manifold $X$ of dimension $3$. Suppose $Sing \mathcal{F}$ consists only of canonical singularities. Consider a holomorphic mapping $f: \bC^{2} \to X$ of generic maximal rank tangent to $\mathcal{F}$. If the canonical line bundle $K_\cF$ of the foliation is big, then $f$ is not Zariski dense.
\end{thmf}

Let us indicate the methods of the proofs, which may be of independent interest. We recall that McQuillan (\cite{McQ0}) showed that one can associate to a transcendental entire curve $f$ a closed positive current of bidimension $(1,1)$  called Ahlfors current.

In the second section, we generalize the construction of such currents for arbitrary non-degenerate holomorphic mappings $f: \bC^p \to X$ in compact K\"ahler manifolds and show how classical Nevanlinna theory (see \cite{GK} or \cite{Sha}) translate into intersection theory for such currents. The problem of associating to several variables holomorphic maps currents with suitable properties has been recently considered by de Th\'elin and Burns-Sibony in \cite{DeT, BuSi}. We refer the reader to these interesting papers for more details on their motivations and for applications in other directions.

The most important part of this theory is developed in the third section which is devoted to the proof of several {\it tautological inequalities}. They are particularly useful when the holomorphic mappings we study are tangent to holomorphic foliations. This is the object of section $4$, where we prove that if the singularities of the foliation are mild, we can control the intersection of the Ahlfors current with the canonical bundle of the foliation. This leads to the theorems stated above proved in section $5$. We illustrate these results, in section $6$, with examples of foliated varieties having the appropriate singularities.

In the final section, we prove a desingularization statement for Ahlfors currents in dimension $3$ which is used to obtain the degeneracy of holomorphic mappings tangent to foliations of general type with canonical singularities.

The present work leads to several questions. 
Two of them seem particularly interesting. 
\begin{question}
Is it possible to remove the hypothesis of the 
existence of local first integrals from Theorem C?
\end{question}
Notice that a positive answer to the previous question when $\dim(X)=3$ would lead, thanks to the work of Cano \cite{Can}, to the non-existence of Zariski dense holomorphic maps from $\bC^2$ into a threefold of general type, which are tangent to a holomorphic foliation. 

\begin{question}
Is it possible to find (numerical/geometrical) conditions insuring the existence of (codimension one) holomorphic foliations on varieties of general type?
\end{question}

\bigskip
\textbf{Acknowledgements}. We would like to thank Michael McQuillan for many interesting discussions on the subject of this paper. We also thank Dominique Cerveau, Charles Favre, Daniel Panazzolo and Fr\'ed\'eric Touzet for useful conversations.

\section{Holomorphic mappings and closed positive currents}
Let a holomorphic mapping $f:\bC^m \to X$ of maximal rank be given, where $X$ is a complex projective manifold. We want to associate to $f$ a closed positive current of bidimension $(1,1)$ adapting in higher dimension the ideas of \cite{McQ0} (see also \cite{Br}) developed in the one-dimensional case.

We fix once and for all a K\"ahler form $\omega$ on $X$. On $\bC^m$ we take the homogeneous metric form
$$\omega_0:=dd^c \ln |z|^2,$$ and denote by 
$$\sigma=d^c \ln |z|^2 \wedge \omega_0^{m-1}$$
the Poincar\'e form.

Consider $\eta \in A^{1,1}(X)$ and for any $r>0$  define
$$T_{f,r}(\eta)=\int_0^r \frac{dt}{t} \int_{B_t} f^*\eta \wedge \omega_0^{m-1},$$
where $B_t \subset \bC^m$ is the ball of radius $t$. Then we consider the positive currents $\Phi_r \in A^{1,1}(X)'$ defined by
$$\Phi_r(\eta):=\frac{T_{f,r}(\eta)}{T_{f,r}(\omega)}.$$
This gives a family of positive currents of bounded mass from which we can extract a subsequence $\Phi_{r_n}$ which converges to $\Phi \in A^{1,1}(X)'.$

Let us prove that 
\begin{claim}
We can choose $\{r_n\}$ such that $\Phi$ is moreover closed.
\end{claim}
\begin{proof}
Take $\beta$ a smooth $(1,0)$ form.
$$T_{f,r}(\overline{\partial}\beta)=\int_0^r \frac{dt}{t} \int_{B_t} f^*\overline{\partial}\beta \wedge \omega_0^{m-1}.$$
$$T_{f,r}(\overline{\partial}\beta)-T_{f,1}(\overline{\partial}\beta)=\int_1^r \frac{dt}{t} \int_{B_t} f^*\overline{\partial}\beta \wedge \omega_0^{m-1},$$
which, by Stokes' formula and then by Fubini's theorem, is equal to
$$\int_1^r \frac{dt}{t} \int_{S_t} f^*\beta \wedge \omega_0^{m-1}=\frac{1}{2}\int_{B_r \setminus B_1} d\ln|z|^2 \wedge f^*\beta \wedge \omega_0^{m-1}.$$ Then Cauchy-Schwartz gives

\begin{eqnarray*}
&& |\frac{1}{2}\int_{B_r \setminus B_1} d\ln|z|^2 \wedge f^*\beta \wedge \omega_0^{m-1}|  \\
&&\leq 
\pi |\int_{B_r \setminus B_1} d\ln |z|^2 \wedge d^c\ln|z|^2\wedge \omega_0^{m-1}|^\frac{1}{2}.|\int_{B_r \setminus B_1} f^*\beta \wedge f^*\overline{\beta} \wedge \omega_0^{m-1}|^\frac{1}{2}\\
&& \leq C \ln(r)^\frac{1}{2}. \left(\int_{B_r \setminus B_1} f^*\omega \wedge \omega_0^{m-1}\right)^\frac{1}{2}\\
&& \leq C \ln(r)^\frac{1}{2}. \left(r \frac{d}{dr}T_{f,r}(\omega)\right)^\frac{1}{2}
\end{eqnarray*}

Since $f$ is of maximal rank, $T_{f,r}(\omega)$ is strictly increasing and has at least a logarithmic growth, therefore
$$\liminf_{r \to +\infty} \frac{r \ln r \frac{d}{dr}T_{f,r}(\omega)}{T_{f,r}(\omega)^2}=0,$$
which concludes the proof.
\end{proof}

\begin{remark}
We remark that all this extends to the setting of meromorphic maps $f: \bC^m \to X$. Indeed, in this case, $f^*\eta$ may have singularities but it has coefficients which are locally integrable.
\end{remark}

\begin{remark}
One remarks that the above definition can easily be generalized to associate currents of any bidimension $(k,k)$, $1 \leq k \leq m$ to $f: \bC^m \to X$. But doing this, one loses for $k \neq 1$ the closedness property as discussed in \cite{DeT} and recently in \cite{BuSi}. Moreover, for our problem, Green-Griffiths' philosophy suggests that the geometry of these holomorphic maps should be determined by a line bundle, $K_X$.
\end{remark}

We denote by $[\Phi] \in H^{n-1,n-1}(X,\bR)$ the cohomology class of $\Phi$.
A consequence of the \textit{First Main Theorem} of Nevanlinna theory (see \cite{Sha} p.63) is (see also \cite[Theorem 3.6]{BuSi} for the same observation).
\begin{lemma}\label{pos}
Let $Z \subset X$ an algebraic hypersurface. If $f(\bC^m) \not\subset Z$ then $[\Phi].[Z] \geq 0.$
\end{lemma}
\begin{proof}
We follow the proof of the First Main theorem given in \cite{Sha} p.61-63. Let $s$ be a section of $\mathcal{O}(Z)$ such that $Z=\{s=0\}$ and fix a metric $h$ on the line bundle $\mathcal{O}(Z)$. Then by the Poincar\'e-Lelong formula
$$dd^c\ln||s||^2=Z-\Theta_h,$$
in the sense of currents, where $\Theta_h$ is the curvature associated to $h$. Taking pullbacks by $f$, assuming $f(0) \not\in Z$, and integrating, we obtain
$$\int_{B_t} dd^c\ln ||s \circ f||^2 \wedge \omega_0^{m-1}=-\int_{B_t} f^*(\Theta_h) \wedge \omega_0^{m-1}+ \int_{f^{-1}(Z)\cap B_t} \omega_0^{m-1}.$$
By Stokes formula, one has
$$\int_{B_t} dd^c\ln ||s \circ f||^2 \wedge \omega_0^{m-1}=\int_{S_t}d^c\ln ||s \circ f||^2 \wedge \omega_0^{m-1},$$
$$\int_0^r\frac{dt}{t} \int_{S_t}d^c\ln ||s \circ f||^2 \wedge \omega_0^{m-1}=\int_{S_r}\ln ||s\circ f||\sigma -\ln ||s\circ f (0)||,$$
where $\sigma$ is the Poincar\'e form.
Therefore
$$T_{f,r}(\Theta_h)=\int_0^r\frac{dt}{t}\int_{f^{-1}(Z)\cap B_t} \omega_0^{m-1}+\int_{S_r}\ln \frac{1}{||s\circ f ||} \sigma + \ln ||s \circ f(0)||.$$
This can be written using the usual notations of Nevanlinna theory (see e.g. \cite{Sha})
$$T_{f,r}(\Theta_h)=N_f(Z,r)+m_f(Z,r)+O(1).$$
We can suppose $||s|| \leq 1$ therefore $m_f(Z,r)\geq 0$ and
$$\Phi(\Theta_h)=\lim_{r_n \to +\infty} \frac{T_{f,r_n}(\Theta_h)}{T_{f,r_n}(\omega)}\geq 0.$$
\end{proof}

\section{Tautological inequalities}

\subsection{The compact case}
Let $X_1:=G(m,TX)$ be the Grassmannian bundle, $\pi: X_1 \to X$ the natural projection, $S_1$ the tautological bundle on $X_1$ and $L:=\overset{m} \bigwedge S_1$. 

For $f:\bC^m \to X$ a holomorphic mapping of maximal rank, we have a natural lifting $f_1: \bC^m \to X_1$ defined by $f_1=(f , [\frac {\partial f}{\partial t_{1}}\wedge \dots\wedge \frac {\partial f}{\partial t_{m}}])$. Then we can associate a closed positive current $\Phi_1$ to $f_1$. 
The tautological inequality becomes

\begin{theorem}\label{taut} With the notation above we have 
$$[\Phi_1].L\geq 0.$$
\end{theorem}

\begin{proof} The K\"ahler form
$\omega$ induces a metric on $L$ of curvature $\Theta$. The map
$\frac {\partial f}{\partial t_{1}}\wedge \dots\wedge \frac {\partial f}{\partial t_{m}}:\bC^m \to \overset{m} \bigwedge T_X$ defines a section $s\in H^0(\bC^m,f_1^*L)$. By the Poincar\'e-Lelong formula, we have:
$$dd^c\ln||s||^2=(s=0)-f_1^*\Theta.$$ Therefore, denoting $\xi:=\left\|\frac {\partial f}{\partial t_{1}}\wedge \dots\wedge \frac {\partial f}{\partial t_{m}}\right\|^2$, we have
$$-T_{f_1,r}(\Theta) \leq \int_0^r\frac{dt}{t} \int_{B_t} dd^c\ln \xi \wedge \omega_0^{m-1},$$ and argueing as in the proof of Lemma \ref{pos}, and using moreover the convexity of the logarithm, we obtain
$$-T_{f_1,r}(\Theta) \leq C+ \frac{m}{2}\ln \int_{S_r} \xi^{\frac{1}{m}}  \sigma.$$
Using the Euclidean metric form $\varphi_0$ which satisfies $\frac{\varphi_0^m}{m!}=r^{2m-1}\sigma \wedge dr$, we have
$$\int_{S_r} \xi^{\frac{1}{m}}  \sigma=\frac{1}{r^{2m-1}}\frac{d}{dr}\frac{1}{m!}\int_{B_r}  \xi^{\frac{1}{m}} \phi_0^m=\frac{1}{m!} \frac{1}{r^{2m-1}} \frac{d}{dr} \left(r^{2m-1} \frac{d\tilde{T}}{dr}\right),$$ where we denote $$\tilde{T}(r):=\int_0^r\frac{dt}{t^{2m-1}}\int_{B_t} \xi^{\frac{1}{m}} \phi_0^m.$$
Now, we use the following classical lemma in Nevanlinna theory (see e.g. \ref{GK}).
\begin{lemma}
Let $F: \bR^+ \to \bR^+$ be a postive, increasing, derivable function. For every $\epsilon > 0$, there exists $E \subset \bR$ satisfying $\int_E dr \leq \int_{1+\epsilon}^{+\infty}\frac{1}{x \ln^{1+\epsilon}(x)}dx<+\infty$ such that for every $x\not\in E,$
$$F'(x)\leq F(x)\ln^{1+\epsilon}(F(x)).$$
\end{lemma}
We apply this lemma to the function $r^{2m-1} \frac{d\tilde{T}}{dr},$ and we obtain that
$$\frac{d}{dr} \left(r^{2m-1} \frac{d\tilde{T}}{dr}\right) \leq r^{2m-1} \frac{d\tilde{T}}{dr}\ln^{1+\epsilon}(r^{2m-1} \frac{d\tilde{T}}{dr}),$$ outside $E$. The lemma applied to $\tilde{T}$ gives
$$\frac{d\tilde{T}}{dr}\leq \tilde{T}(r)\ln^{1+\epsilon}(\tilde{T}(r)),$$
outside $E$.
Now, a consequence of the classical Hadamard inequality is
$$\xi^{\frac{1}{m}} \phi_0^m \leq f^*\omega \wedge \phi_0^{m-1}.$$ Therefore we obtain
$$\tilde{T}(r)\leq \int_0^r\frac{dt}{t^{2m-1}}\int_{B_t} f^*\omega \wedge \phi_0^{m-1}=T_{f,r}(\omega).$$
Finally, these inequalities give
\begin{equation}\label{intaut}
-T_{f_1,r}(\Theta) \leq O(\ln(T_{f,r}(\omega)),
\end{equation}
outside $E$, which proves the theorem.
\end{proof}

\subsection{The logarithmic case}
Let $X$ be a smooth projective variety, $D \subset X$ a simple normal crossing divisor, $\omega$ a K\"ahler metric on $X$ and $f: \bC^m \to X$ a holomorphic mapping of maximal rank such that $f(\bC^m) \not \subset D$. We set  $X_1:=\bP(\overset{m} \bigwedge  T_X(-\log(D)))$. We have a natural map $f_1: \bC^m \to X_1$ and a tautological line bundle $L:=\mathcal{O}_{X_1}(1)$.

\begin{theorem}\label{logtaut}

$T_{f_1,r}(L) \leq N^1_{f}(D,r)+O(\ln T_{f,r}(\omega)) \|$. 

\end{theorem}
(As usual the symbol $\|$ means that the inequality holds outside a set of finite Lesbesgue measure). 
\begin{proof}
We follow the approach given in \cite{Gas} for the $1$-dimensional case.
We write $D=\sum D_i$. Let $s_i$ be sections of the $D_i$ and we choose hermitian metrics on the associated line bundles. We consider a smooth metric on $T_X(-\log(D))$ induced by a singular $(1,1)$-form
$$\omega^{sm}:=A\omega+ \sum \frac{d ||s_i|| \wedge d^c ||s_i||}{||s_i||^2}.$$

We also consider a singular hermitian metric on $T_X(-\log(D))$ induced by

$$\tilde{\omega}:= \omega + \sum \frac{d ||s_i|| \wedge d^c ||s_i||}{||s_i||^2(\ln ||s_i||)^2}.$$

The map $f_1$ together with the map $\overset{m} \bigwedge T_{\bC^m} \to \mathcal{O}_X$ given by $\frac{\partial}{\partial t_1}\wedge \dots \wedge \frac{\partial}{\partial t_m} \to 1$ gives a map $g: \bC^m \to Y:=\bP(\overset{m} \bigwedge T_X(-\log(D)) \oplus \mathcal{O}_X).$ Denote by $M$ the tautological line bundle over $Y$.

We have an inclusion $X_1 \subset Y$ and $\mathcal{O}_{Y}(X_1)=M.$ The forms $\omega^{sm}$ and $\tilde{\omega}$ induce respectively  smooth and singular metrics on $M$ and $L$.

Now, we apply the First Main Theorem to $g$ and $M$ with respect to the singular metric:
$$T_{g,r}(c_1(M)^s)=N_{g}(X_1,r)+m_{g}(X_1,r)+O(1).$$ The image of 
$g$ meets $X_1$ only over $D$ with multiplicity at most $1$, therefore
$$N_{g}(X_1,r) \leq N^1_{f}(D,r).$$
We consider the blow-up $p: Z \to Y$ along the zero section of $\overset{m} \bigwedge T_X(-\log(D))$. Then we have a holomorphic map $q: Z \to X_1$ and
$$q^*L=p^*M(-E)$$ where $E$ is the exceptional divisor in $Z$. Therefore we obtain
$$T_{f_1,r}(c_1(L)^s)\leq N^1_{f}(D,r)+m_{g}(X_1,r)-m_{g}(E,r)+O(1).$$
Comparing the two metrics, we see that
$$T_{f_1,r}(c_1(L)^{sm})\leq T_{f_1,r}(c_1(L)^s) + O(\ln T_{f,r}(\omega)).$$
We can suppose $m_{g}(E,r)\geq 0$ so to finish the proof we just have to bound 
$$m_{g}(X_1,r)=\int_{S_r} \ln \frac{1}{||s \circ f ||} \sigma,$$
where $s$ is a holomorphic section of $\mathcal{O}_{Y}(X_1)$. Let 
$$\xi:=f^*(\tilde{\omega}^m)\left(\frac {\partial }{\partial t_{1}}\wedge \dots\wedge \frac {\partial }{\partial t_{m}}\right)$$
where $\tilde{\omega}^m$ is the metric induced on $\overset{m} \bigwedge T_X(-\log(D))$ by $\tilde{\omega}$. To conclude we need to find an upper bound for $$\int_{S_r} \ln{(\xi)} \sigma.$$ Following the end of the proof of Theorem \ref{taut}, we are reduced to find a bound for
$$\int_0^r\frac{dt}{t}\int_{B_t} f^*\tilde{\omega} \wedge \omega_0^{m-1}.$$
Recall that we have the following fomula
$$dd^c\ln(\ln||s_i||^2)=\frac{dd^c \ln ||s_i||^2}{\ln ||s_i||^2}-\frac{d \ln ||s_i||^2 \wedge d^c \ln ||s_i||^2}{(\ln ||s_i||^2)^2}.$$ Therefore 
$$\tilde{\omega} \leq C\omega-\sum dd^c\ln(\ln^2||s_i||)$$
for a constant $C$, and finally, by Stokes formula
$$\int_0^r\frac{dt}{t}\int_{B_t} f^*\tilde{\omega} \wedge \omega_0^{m-1} \leq O(\ln T_{f,r}(\omega))-\sum \int_{S_r} \ln(\ln^2||s_i||)  \sigma.$$

Since we may assume $||s_i||<\delta<1$, the theorem is proved.
\end{proof}

\begin{remark}
We remark that, in the particular case $m= \dim X$, we recover the second main theorems of Griffiths-King \cite{GK}.
\end{remark}

\subsection{Refined inequalities}
In the case $m=1$ of entire curves, McQuillan \cite{McQ0} shows that one can include in the tautological inequality the defect with respect to a finite number of reduced points. Following the approach of \cite{Siu04} to this question, we would like to show that in our situation the previous inequalities can be made more precise by including the defect with respect to submanifolds of codimension at least $2$.

\begin{lemma}\label{bup}
Let $\Delta \subset \bC^n$ be an $n$-dimensional disc with holomorphic coordinates $z_1,\dots, z_n$, $V \subset \Delta$ be the locus $z_{k+1}=\dots=z_n=0$ and $\pi: \tilde{\Delta} \to \Delta$ be the blow-up of $\tilde{\Delta}$ along $V$. Let $E=\pi^{-1}(V)$ be the exceptional divisor. Then $$\pi^*\left(\Omega^p_{\Delta}\right) \subset \Omega^p_{\tilde{\Delta}}(\log E) \otimes \mathcal{I}_E^{p-k},$$
for $p \geq k+1$.
\end{lemma}

\begin{proof} Notice that 
$\tilde{\Delta} \subset \Delta \times \bP^{n-k-1}$ is defined by $\tilde{\Delta}=\{(z,l) : z_il_j=z_jl_i, k+1 \leq i,j \leq n\}$. In the affine chart defined by $l_j \neq 0$ we take the coordinates $z_i, z_j, \frac{l_q}{l_j}$ for $1\leq i \leq k$, $k+1 \leq q \leq n$ and $q\neq j$. The exceptional divisor $E$ is defined by $z_j=0$. We have $\pi^*(dz_i)=dz_i$ for  $1\leq i \leq k$,  $\pi^*(dz_j)=dz_j=z_j\left(\frac{dz_j}{z_j}\right)$ and $\pi^*(dz_q)=z_j d\left(\frac{l_q}{l_j}\right)+\frac{l_q}{l_j}dz_j=z_j\left(d\left(\frac{l_q}{l_j}\right)+\frac{l_q}{l_j}\frac{dz_j}{z_j}\right).$ Now take $\omega:=dz_{i_1}\wedge \dots \wedge dz_{i_p}$, where $p \geq k+1$. From the previous relations, it is obvious that $\pi^*\omega$ is a differential $p$-form with logarithmic poles on $E$ whose coefficients vanish on $E$ to order at least $p-k$.
\end{proof}

The next tool we need is the Lemma on logarithmic derivatives. For a meromorphic function on $\bC$ it  was done by Nevanlinna \cite{Nev39}, and then generalized on $\bC^m$ by Vitter \cite{Vi77} (see also \cite{No85} and \cite{Yam} for other generalizations).

\begin{theorem}\label{derlog}
Let $f$ be a meromorphic function on $\bC^m$. Then
$$\int_{S_r} \ln^+ \left |\frac{\frac{\partial}{\partial t_i}f}{f}\right | \sigma \leq O( \ln T_{f,r}+ \ln r).$$
\end{theorem}

Now, we can prove the following

\begin{theorem}
Let $H$ be an ample line bundle on a projective manifold $X$ of dimension $n$. Let $Z \subset X$ be a submanifold of codimension $n-k \geq 2$. Let $f: \bC^m \to X$ be a holomorphic map of maximal rank. Let $\alpha$ be a positive rational number and $l$ be a positive integer such that $\alpha l$ is an integer. Let $\sigma \in H^0(X,S^l \Omega_X^m \otimes (\alpha l H))$ such that $f^*\sigma$ is not identically zero. Let $W$ be the zero divisor of $\sigma$ in $X_1:=\bP(\overset{m} \bigwedge  T_X)$. Then
\begin{equation}\label{reftaut}
\frac{1}{l}N_{f_1}(W,r)+(m-k)m_f(Z,r) \leq \alpha T_{f,r}(H) +O(\ln T_{f,r}(H)+ \ln r) ||.
\end{equation}
\end{theorem}

\begin{proof}
Let $\pi: \tilde{X} \to X$ be the blow-up of $Z$ and $E=\pi^{-1}(Z)$. Let $\tilde{f}: \bC^m \to \tilde{X}$ be the lifting of $f$ and let $\tau=\pi^*\sigma$. By lemma \ref{bup}, $\tau$ is a holomorphic section of $S^l\Omega^m_{\tilde{X}}(\log E)\otimes \pi^*(\alpha l H)$ which vanishes to order at least $l(m-k)$ on $E$. Let $s_E$ be the canonical section of $E$. Let $\tilde{\tau}=\frac{\tau}{s_E^{l(m-k)}}$ which is a holomoprhic section of $S^l\Omega^m_{\tilde{X}}(\log E)\otimes \pi^*(\alpha l H)\otimes (-l(m-k) E)$ over $\tilde{X}$. We consider $h_E$ a smooth hermitian metric on $E$ and $h_H$ a smooth hermitian metric on $H$.
\begin{eqnarray*}
2l(m-k) m_{\tilde{f}}(E,r)&=&\int_{S_r}\ln^+ \frac{1}{||s_E^{l(m-k)} \circ \tilde{f} ||_{h_E^{l(m-k)}}^2} \sigma  +O(1)\\
&\leq& \int_{S_r} \ln^+ \frac{1}{||\tau \circ f ||_{h_H^{\alpha l}}^2}\sigma + \int_{S_r}\ln^+ ||\tilde{\tau} \circ \tilde{f} ||_{h_H^{\alpha l}h_E^{-l}}^2\sigma +O(1)\\
&=&-\int_{S_r}\ln ||\tau \circ f ||_{h_H^{\alpha l}}^2 \sigma +\int_{S_r}\ln^+||\tau \circ f ||_{h_H^{\alpha l}}^2 \sigma + \int_{S_r}\ln^+ ||\tilde{\tau} \circ \tilde{f} ||_{h_H^{\alpha l}h_E^{-l}}^2\sigma +O(1).
\end{eqnarray*}
Taking the logarithms of global meromorphic functions as local coordinates, the lemma on the logarithmic derivative \ref{derlog} implies
$$\int_{S_r}\ln^+||\tau \circ f ||^2 \sigma = O( \ln T_{f,r}(H)+ \ln r),$$
$$\int_{S_r}\ln^+ ||\tilde{\tau} \circ \tilde{f} ||^2\sigma=O( \ln T_{f,r}(H)+ \ln r).$$
Therefore, by the First Main Theorem, we obtain
$$2l(m-k) m_{\tilde{f}}(E,r) \leq -2N_{f_1}(W,r)+2\alpha l T_{f,r}(H)+2lT_{f_1,r}(\mathcal{O}_{X_1}(1))+O( \ln T_{f,r}(H)+ \ln r).$$
The inequality \ref{intaut} implies that $T_{f_1,r}(\mathcal{O}_{X_1}(1))=O( \ln T_{f,r}(H)+ \ln r)$ and this concludes the proof.
\end{proof}

\begin{corollary}
Let $X$ be a projective manifold of dimension $n$, $H$ an ample line bundle on $X$, $Z \subset X$ a submanifold of codimension $n-k \geq 2$. Let $f: \bC^m \to X$ a holomorphic map of maximal rank. Then
$$T_{f_1,r}(\mathcal{O}_{X_1}(1)) +(m-k)m_f(Z,r) \leq O( \ln T_{f,r}(H)+ \ln r).$$
\end{corollary}

\begin{proof}
Choosing $\alpha$ sufficiently large, $\mathcal{O}_{X_1}(1)\otimes (\alpha H)$ is ample over $X_1$. Then if we choose $l$ sufficently large, $\mathcal{O}_{X_1}(l)\otimes (\alpha l H)$ is very ample. Then there exists $\sigma \in H^0(X, S^l \Omega_X^m \otimes (\alpha l H))$ such that the defect under the mapping $f_1$ $$\liminf_r  \frac{m_{f_1}(W,r)}{T_{f_1,r}(\mathcal{O}_{X_1}(l)\otimes (\alpha l H))}=0,$$ where $W \subset X_1$ is the zero divisor of $\sigma$. Therefore, from inequality \ref{reftaut}, we deduce
$$T_{f_1,r}(\mathcal{O}_{X_1}(1)) +(m-k)m_f(Z,r) \leq O( \ln T_{f,r}(H)+ \ln r).$$
\end{proof}

\section{Holomorphic mappings and foliations}
\subsection{The smooth case}
A smooth foliation of dimension $p$ on $X$ is given by an integrable subbundle  $\mathcal{F} \subset T_X$ of rank $p$. So we have an exact sequence
$$0 \to  \mathcal{F} \to T_X \to N_\mathcal{F} \to 0,$$ 
where $N_\mathcal{F}$ is called the normal bundle of the foliation. The line bundle $K_\mathcal{F}:=\det(\mathcal{F}^*)$ is the canonical bundle of the foliation. Notice that we have an isomorphism $K_X=K_\mathcal{F} \otimes \det(N_\mathcal{F}^*).$

\begin{proof}[Proof of Theorem A]
Let   $f: \bC^{p} \to X$ be a holomorphic map of generic maximal rank which is a leaf of the foliation  $\mathcal{F}$, i.e. it is tangent to  $\mathcal{F}$.
The foliation $\mathcal{F}$ defines a section $F \subset X_1:=\bP(\overset{p} \bigwedge  T_X)$ over $X$ and the tautological line bundle verifies $L|_{F}=\pi^*K_\mathcal{F}^{-1}$. From Theorem \ref{taut} we have
$$0\leq [\Phi_1].L=[\Phi].K_\mathcal{F}^{-1}=[\Phi].K_X^{-1}-[\Phi].\det(N_\mathcal{F}).$$
Suppose that the image of $f$ is not contained in $\NAmp(K_X)$. Then there exists a decomposition of the canonical divisor
$K_X\sim_\Q A + E$ into the sum of a $\bQ$-ample divisor $A$ and   $\bQ$-effective divisor $E$ such that the image of $f$ is not contained in $\Supp(E)$.
By the smoothness assumption the normal bundle $N_\mathcal{F}|_{\mathcal{F}}$ is flat, we obtain using Lemma \ref{pos}
$$[\Phi].K_X^{-1}-[\Phi].\det(N_\mathcal{F})=-[\Phi].K_X= -[\Phi].A - [\Phi].E<0,$$
which is a contradiction.
\end{proof}
Notice that when the foliation is not smooth the control of its normal bundle is a delicate point. 
\subsection{The singular case}
In general, there is no reason that the foliations we are working with is non-singular. Therefore, it is important to generalize the preceding result to singular foliations.
We consider  singular holomorphic foliations $\mathcal{F}$ of codimension one on a projective manifold $X$ of dimension $n$. It is locally given by a differential equation $\omega$ where
$$\omega=\sum_{i=1}^{n} a_i(z)dz_i$$
is an integrable $1-$ form, i.e. $\omega \wedge d\omega = 0$, and the coefficients $a_i$ have no common factor. The \textit{singular locus} $Sing \mathcal{F}$ is locally given by the common zeros of the coefficients $a_i$.

The foliation can be defined by a collection of $1$-forms $\omega_j \in \Omega_X^1(U_j)$ such that $\omega_i=f_{ij} \omega_j$ on $U_i\cap U_j$, $f_{ij} \in \mathcal{O}_X^*(U_i\cap U_j)$.

As in the smooth case, there are two holomorphic line bundles associated to $\mathcal{F}$, the normal bundle $N_\mathcal{F}$ and the canonical bundle $K_\mathcal{F}$ of the foliation $\mathcal{F}$. $N_\mathcal{F}$ is defined by the cocycle $f_{ij}$ and $K_\mathcal{F}=K_X\otimes N_\mathcal{F}$.

In the case of surfaces, one uses the theorem of resolution of singularities of Seidenberg \cite{Sei} to work only with reduced singularities. In general, one may expect theorems of resolution of singularities to reduce the problem to {\it canonical singularities} as introduced by McQuillan \cite{McQ08} following the approach of the Mori program.

\begin{definition}
Let $(X,\mathcal{F})$ be a pair where $X$ is a projective variety and $\mathcal{F}$ a foliation.
Let $p: (\tilde{X}, \tilde{\mathcal{F}}) \to (X,\mathcal{F})$ be a birational morphism of pairs. We can write
$$K_{\tilde{\mathcal{F}}}=p^*K_{\mathcal{F}}+ \sum a(E,X,\mathcal{F})E.$$
$a(E,X,\mathcal{F})$ is independent of the morphism $p$ and depends only on the discrete valuation that corresponds to $E$. It is called the discrepancy of $(X,\mathcal{F})$ at $E$.
We define
\begin{eqnarray*}
discrep(X,\mathcal{F})&=&\inf \{a(E,X,\mathcal{F}); E \text{ corresponds to a discrete valuation}\\
\text{such that } Center_X(E)\neq \emptyset &and& \codim(Center_X(E)) \geq 2\}.
\end{eqnarray*}

We say that $(X,\mathcal{F})$ has canonical singularities if $discrep(X,\mathcal{F}) \geq 0.$
\end{definition}

Following McQuillan \cite{McQ08}, we may generalize this definition to  singularities of foliated pairs: consider  a triple $(X, \mathcal{F}, B)$ where $X$ is a projective variety, $\mathcal{F}$ is a foliation on it and $B=\sum_i\left(1-\frac{1}{m_i}\right) B_i$  is a boundary $\Bbb Q$--Cartier divisor, where $m_i \in \bN \cup \{\infty\}$.  Consider the following function on Cartier divisors of $X$:

$$\epsilon(H):= \left\{
\begin{array}{l}
0 \text{ if } H \text{ is invariant by the foliation } \\
1 \text{ if } H \text{ is not invariant by the foliation}.
\end{array}
\right. $$

For every $p: (\tilde{X}, \tilde{\mathcal{F}}) \to (X,\mathcal{F})$ blow up with (reduced) exceptional divisor $E=\sum_jE_j$,  write $p^\ast(B_i)=\sum_j\nu_i(E_j)E_j
+B_i'$, where $B_i'$ is the strict transform of $B_i$.

\begin{definition}\label{def:sing}
$(X, B, \mathcal{F})$ is said to be
\begin{enumerate}
\item $\text{terminal if } a(E_j,X,\mathcal{F}) - \sum_i \left(1-\frac{1}{m_i}\right) \epsilon_i \nu_i(E_j) >0,$
\item $\text{canonical if } a(E_j,X,\mathcal{F}) - \sum_{ij} \left(1-\frac{1}{m_i}\right) \epsilon_i \nu_i(E_j) \geq 0,$
\item $\text{log terminal if } a(E_j,X,\mathcal{F}) - \sum_i \left(1-\frac{1}{m_i}\right) \epsilon_i \nu_i(E_j) >-\epsilon(E_j),$
\item $\text{log canonical if } a(E_j,X,\mathcal{F}) - \sum_i \left(1-\frac{1}{m_i}\right) \epsilon_i\nu_i(E_j) \geq-\epsilon(E_j),$
\end{enumerate}
for every blow up $p: (\tilde{X}, \tilde{\mathcal{F}}) \to (X,\mathcal{F})$, where  $\epsilon_i:=\epsilon(B_i)$ and $\nu_i(E_j)$ are as above.
\end{definition}

\begin{remark}
Up to now, there are few theorems of resolution of singularities: it is known for codimension $1$ foliations in dimension $3$ \cite{Can} and recently for foliations by curves in dimension $3$ \cite{McQPan}.
\end{remark}

First, we consider the situation where we have a singular codimension $1$ foliation with \textit{logarithmic simple singularities} i.e. we can write $\omega$ in local coordinates
\begin{equation}\label{eq:logsing}
\omega=\left(\prod_{i=1}^r z_i \right) \sum_{i=1}^r \lambda_i\frac{dz_i}{z_i},
\end{equation}
where $\sum_{i=1}^r m_i\lambda_i \neq 0$, for any non-zero vector $(m_i) \in \bN^r.$

In particular, the only integral hypersurfaces are the components of $z_1\dots z_r=0.$
We can suppose, up to doing some blow ups, that the foliation have simple singularities adapted to a normal crossing divisor \cite{Can}. This means that we have an invariant simple normal crossing divisor $E$ such that every $P \in Sing \mathcal{F}$ belongs to at least $r-1$ irreducible components of  $E$.

We consider $X_1:=\bP(\overset{m} \bigwedge  T_X(-\log(E)))$ with the natural projection $\pi: X_1 \to X$. 

The foliation $\mathcal{F}$ defines a section $F \subset X_1$ over $X$ and

$$\mathcal{O}_{X_1}(-1)|_F=\pi^*K_\mathcal{F}^{-1}.$$
Fix an ample divisor $H$ on $X$. 
\begin{prop}\label{Edoesntcount}
With the notation above we have 
$$\lim_{r \to +\infty} \frac{N^1_{f}(E,r)}{T_{f,r}(H)}=0.$$
\end{prop}

In order to prove  Proposition \ref{Edoesntcount} we will prove the following Lemma:

\begin{lemma}\label{formallemma1} 
Let $(X, H)$ be a smooth polarized projective variety and $C$ a smooth closed subvariety of codimension two. Let $X_C$ be the  formal completion of $X$ around $C$ and $\iota:V\hookrightarrow X_C$ be a smooth formal subvariety containing $C$ of codimension one. Then we can find a sequence of global sections $s_m\in H^0(X,H^{\otimes m})$ such that $\iota^\ast(s_m)\geq q_mC$ with
$$\lim_{m\to\infty}{\frac{m}{q_m}}=0.$$
\end{lemma}

\begin{proof} Observe that $C$ is a divisor in $V$. Denote by $N_C$ the normal line bundle of $C$ inside $V$and by $C_i$ the $i$--th formal neigborhood of $C$ inside $V$. For every positive integers $i\geq 1$ and $m\geq 1$, we have an exact sequence
$$0\to H^{\otimes m}\vert_C\otimes N_C^{\otimes -(i-1)}\longrightarrow H^{\otimes m}\vert_{C_i}\longrightarrow H^{\otimes m}\vert_{C_{i-1}}\to 0.$$
Denote by $E_m^i$ the kernel of the composite map 
$$H^0(X,H^{\otimes m})\longrightarrow H^0(V,\iota^\ast(H^{\otimes m}))\longrightarrow H^0(C_i,H^{\otimes m}).$$ 
Fix $\epsilon>0$ sufficiently small. We will prove that, for $m\gg 0$ we have that $E_m^{m^{1+\epsilon}}\neq \{0\}$, and this will be enough to conclude. 

The snake lemma applied to the exact sequence above gives rise to an inclusion
$$\gamma_m^i :E_m^{i-1}/{E_m^{i}}\hookrightarrow H^0(C, H\vert_C^{\otimes m}\otimes N_C^{-(i-1)}).$$
Consequently 
$$\dim(E_m^{i})\geq \dim(E_m^{i-1})-h^0(C, H\vert_C^{\otimes m}\otimes N_C^{-(i-1)}).$$
By Riemann--Roch Theorem (or Hilbert--Samuel Theorem) we can find constants $A$ ans $A_1$ independent of $m$ and $i$ such that $h^0(X, H^{\otimes m})\geq Am^n$ and $h^0(C, H^{\otimes m}\otimes N_C^{-i})\leq A_1(m+i)^{n-2}$. Thus we obtain
$$\dim(E_m^{m^{1+\epsilon}})\geq Am^n-\sum_{i=1}^{m^{1+\epsilon}}A_1(m+(i-1))^{n-2}.$$
Since, for a suitable $A_2$ independent on $m$ we have that, for $m\gg 0$, 
$$\sum_{i=1}^{m^{1+\epsilon}}A_1(m+(i-1))^{n-2}\leq\int_1^{m^{1+\epsilon}}(m+(t-1))^{n-2}dt\leq A_2m^{(n-2)(1+\epsilon)},$$ the conclusion follows.

\end{proof}

Let us show how lemma \ref{formallemma1} implies lemma Proposition \ref{Edoesntcount}.
\begin{proof}[Proof of Proposition \ref{Edoesntcount}]
Let $E_1$ be an irreducible component of $E$. The two subvarieties $f(\bC^{n-1})$ and $E_1$ are both invariant for the foliation. Let $V$ be the leaf containing $f(\bC^{n-1})$. We  may suppose that $E_1$ and $V$ intersect  properly on an irreducible $C$, component of $Sing\mathcal{F}$ which is  a smooth closed subvariety of codimension two of $X$. We apply Lemma \ref{formallemma1} and we obtain
$$N^1_{f}(E_1,r)=N^1_{f}(C,r) \leq \frac{1}{q_m}N_{f}(\{s_m=0\},r)\leq \frac{m}{q_m}T_{f,r}(H)+c_m(H)$$
where $c_m(H)$ is a constant depending on $f$, $H$ and $s_m$ but independent on $r$. Lemma \ref{Edoesntcount} follows once one divide the inequality above by $T_{f,r}(H)$ and let $m$ and $r$ tend to infinity. 
\end{proof}

Now, we use Theorem \ref{logtaut} to obtain as an immediate consequence the following.

\begin{theorem}\label{canineq}
Let $\mathcal{F}$ be a holomorphic foliation of codimension one on a projective manifold $X$, $\dim X=n$. Suppose $Sing \mathcal{F}$ consists only of logarithmic simple singularities. Consider a holomorphic mapping $f: \bC^{n-1} \to X$ of generic maximal rank tangent to $\mathcal{F}$ which is Zariski dense. Then
$$[\Phi].K_\mathcal{F}^{-1} \geq 0.$$
\end{theorem}

This result implies Theorem B of the introduction.
\begin{proof}[Proof of Theorem B]
Let   $f: \bC^{n-1} \to X$ be a holomorphic map of generic maximal rank which is a leaf of the foliation  $\mathcal{F}$, i.e. it is tangent to  $\mathcal{F}$.
Suppose that the image of $f$ is not contained in $\NAmp(K_{\mathcal{F}})$. Then there exists a decomposition of the canonical divisor of the foliation
$K_{\mathcal{F}}\sim_\Q A + E$ into the sum of a $\bQ$-ample divisor $A$ and   $\bQ$-effective divisor $E$ such that the image of $f$ is not contained in $\Supp(E)$.
Using Lemma \ref{pos} we get
$$[\Phi].K_{\mathcal{F}}^{-1} =-[\Phi].A - [\Phi].E<0,$$
which is in contradiction with Theorem \ref{canineq}.
\end{proof}

\section{Applications}

We deal with foliations with canonical singularities and local holomorphic first integrals i.e. the form $\omega$ can be written 
\begin{equation}\label{eq:1int}
\omega=g df
\end{equation}
 with $g$ and $f$ holomorphic and $g$ nonvanishing.

\begin{proof}[Proof of Theorem C]
Suppose we have such a map $f$.
Since we have local first integrals, by Hironaka's theorem \cite{Hir}, we have a resolution $\pi:(\tilde{X}, \tilde{\mathcal{F}}) \to (X,\mathcal{F})$ where $(\tilde{X}, \tilde{\mathcal{F}})$ has only logarithmic simple singularities. Consider the (meromorphic) lifting $\tilde{f}: \bC^{n-1} \to \tilde{X}.$ Then Theorem \ref{canineq} gives
$$[\tilde{\Phi}].K_{\tilde{\mathcal{F}}}^{-1} \geq 0.$$ 
Since we have canonical singularities we have
$$K_{\tilde{\mathcal{F}}} \geq \pi^*K_\mathcal{F}$$ and therefore
\begin{equation}\label{eq:1}
[\Phi].K_\mathcal{F}^{-1} \geq 0.
\end{equation}

The existence of local integrals $\omega=g df$ implies that $d\omega=\beta \wedge \omega$, where $\beta =\frac{dg}{g}$ is a holomorphic $1$-form.

Let us take an open covering $\{U_j\}_{j\in I}$ of $X$ and holomorphic one-forms $\omega_j \in \Omega_X^1(U_j)$ generating $\mathcal{F}$ such that 
$$d\omega_j=\beta_j \wedge \omega_j.$$ 
On each $U_i\cap U_j$ we also have 
$$\omega_i= g_{ij}\omega_j$$
where the cocycle $\{g_{ij}\}$ defines $N_\mathcal{F}$. So, we have
$$\beta_i \wedge \omega_i=d\omega_i=dg_{ij}\wedge \omega_j +g_{ij}d\omega_j=\left(\frac{dg_{ij}}{g_{ij}}+\beta_j\right) \wedge \omega_i,$$
and therefore
$$\left(\frac{dg_{ij}}{g_{ij}}+\beta_j - \beta_i\right) \wedge \omega_i=0.$$
We can find smooth $(1,0)$-forms $\gamma_j \in A^{1,0}(U_j)$ such that
$$\gamma_j \wedge \omega_j = 0,$$
on $U_j,$ and 
$$\frac{dg_{ij}}{g_{ij}}= \beta_i - \beta_j+\gamma_i-\gamma_j,$$
on $U_i\cap U_j$.
The $2$-form defined by
$$\frac{1}{2i\pi}d(\beta_j + \gamma_j)$$
on $U_j$ represents the first Chern class of $N_\mathcal{F}$.

The relation $d\omega_j=\beta_j \wedge \omega_j$ implies that $d\beta_j|_{\mathcal{F}}\equiv 0.$
Moreover $d\gamma_j|_{\mathcal{F}}\equiv 0.$
Therefore we obtain
\begin{equation}\label{eq:2}
[\Phi].N_\mathcal{F}=0.
\end{equation}

Suppose that the image of $f$ is not contained in $\NAmp(K_X)$. Then there exists a decomposition of the canonical divisor
$K_X\sim_\Q A + E$ into the sum of a $\bQ$-ample divisor $A$ and   $\bQ$-effective divisor $E$ such that the image of $f$ is not contained in $\Supp(E)$.
Using Lemma \ref{pos}, (\ref{eq:1}) and (\ref{eq:1})
 we conclude thanks to the following  contradiction
$$0<[\Phi].A + [\Phi].E=[\Phi].K_X=[\Phi].K_\mathcal{F}+[\Phi].N_\mathcal{F}^*=[\Phi].K_\mathcal{F}\leq 0.$$
\end{proof}

When $Sing \mathcal{F}$ has codimension $\geq 3$ we can use the following theorem due to Malgrange \cite{Mal}.
\begin{theorem}\label{thm:Mal}
Let $\mathcal{F}$ be a germ of foliation at $(\bC^n,0)$.
If $\codim Sing \mathcal{F} \geq 3$ then $\mathcal{F}$ has a holomorphic first integral.
\end{theorem}
We can now give the (immediate) proof of the corollaries stated in the introduction. 
\begin{proof}[Proof of Corollary D]
It follows from Theorems C and \ref{thm:Mal}.
\end{proof}

\begin{proof}[Proof of Corollary E] Let $X_d$ (respectively ${\mathcal{F}}_d$)
be as in the statement. 
We have $K_{X_d}=\mathcal{O}_{X_d}(d-n-2)$ (respectively $K_{\mathcal{F}_d}=\mathcal{O}_{\bP^{n}}(d-n+1)$). Therefore
item (1) follows from Theorem C and (\ref{eq:empty}) and
item (2) follows from Theorem B and (\ref{eq:empty}).
\end{proof}

\section{Examples}

When a foliation has a holomorphic first integral, it is canonical if and only if the level sets are log--canonical:

\begin{proposition}\label{lc}
Let $(X,{\mathcal{F}})$ be a codimension one foliation. Suppose that, locally on $X$,  the conormal bundle of ${\mathcal{F}}$ is generated by $\omega=d(f)$; suppose that the divisor $D:=\{ f=0\}$ is log--canonical. Then the foliation ${\mathcal{F}}$ is canonical.
\end{proposition}

\begin{proof} Let $\pi:(\tilde{X}, \tilde{D})\to(X,D)$ be a Hironaka resolution of the pair $(X,D)$ and let $E_1,\dots, E_r$ be the divisors contracted by $\pi$. By construction $\tilde{D}$ is the strict transform of $D$. Denote by $K_X$, resp. $K_{\tilde {X}}$, resp. $K_D$, resp. $K_{\tilde {D}}$ the canonical sheaf of $X$, resp. $\tilde{X}$, resp. $D$, resp $\tilde{D}$. By construction, there are positive constants $b_i$ and $r_i$ such that $K_{\tilde{X}}=\pi^\ast(K_X)+\sum_ib_iE_i$ and $\pi^\ast(D)=\tilde D+\sum_i r_iE_i$.

By adjunction formula we get $K_{\tilde{D}}=\pi^\ast(K_D)+\sum_i(b_i-r_i)E_i\vert_{\tilde{D}}$. Since $D$ is log--canonical, we have that $b_i-r_i\geq -1$.

Let $k(X)$ be the function field of $X$ and $R\subset k(X)$ be a discrete valuation ring with fraction field $k(X)$. Let $p$ be a uniformizer of $R$. Suppose that $D$ is locally given by $p^N u$ with $u\in R^\ast$. Thus the restriction of $\omega$ to $\rm{Spec}(R)$ is $(N-1)p^{N-1}ud(p)+p^Nd(u)$. This implies the following: Denote by $N_{\tilde{\mathcal{F}}}$ the normal line bundle of the foliation induced by $\mathcal{F}$ on $\tilde{X}$; then
$$\pi^\ast(N_{\mathcal{F}})=N_{\tilde{\mathcal{F}}}+\sum_i(r_i-1)E_i.$$
Consequently, since $K_X=K_{\mathcal{F}}-N_{\mathcal{F}}$ and $K_{\tilde{X}}=K_{\tilde{\mathcal{F}}}-N_{\tilde{\mathcal{F}}}$, a straightforward calculation gives
$$K_{\mathcal{F}}-\pi^\ast(K_{\mathcal{F}})=\sum_i(b_i-r_i)E_i+\sum_iE_i.$$
The conclusion follows.

\end{proof}

Now, let us consider a ramified cover $\pi: Y \to X$ and the induced foliation on $Y$, $\mathcal{G}:=\pi^*\mathcal{F}$. We can write
$$K_Y:=\pi^*(K_X+\Delta),$$
where $\Delta=\sum_i\left(1-\frac{1}{m_i}\right) Z_i$.

Then we have
\begin{lemma}\label{ramcan}
If $(X,\Delta, \mathcal{F})$ has canonical singularities then $(Y,\mathcal{G})$ has canonical singularities.
\end{lemma}
\begin{proof}
We have
$$K_\mathcal{G}=\pi^*K_\mathcal{F}+\sum_i\epsilon(Z_i)\pi^*\left(1-\frac{1}{m_i}\right)Z_i.$$

Therefore, if $E_Y$ is an exceptional divisor over $Y$ dominating $E_X$ with multiplicity $r$, we have for the corresponding discrepancies
$$a(E_Y)=r(a(E_X)-\sum_i \left(1-\frac{1}{m_i}\right)\epsilon_i\nu_i(E_X))+\epsilon(E_X)(r-1).$$
\end{proof}

The following result describes locally how one can produce foliations with canonical singularities.
\begin{proposition}
Let $(\mathcal{U},\mathcal{F})$ be a germ of a smooth foliation on a complex manifold of dimension $n \geq 3$ and $D \subset U$ a smooth divisor. Let $\pi: (\mathcal{V},\mathcal{G}) \to (\mathcal{U},\mathcal{F})$ be a covering ramified along $D$ where $\mathcal{G}:=\pi^*\mathcal{F}$. Then
\begin{enumerate}
\item If $\mathcal{F}$ is transverse to $D$ then $\mathcal{G}$ is smooth.
\item If $D$ has only isolated non-degenerate quadratic-type tangencies with $\mathcal{F}$ (i.e. if $D=(h=0)$ then the restriction of $h$ to a leaf has only isolated non-degenerate critical points), then $\mathcal{G}$ has isolated canonical singularities.
\item If the degree of the covering $\pi$ is  $2$ and $D$ has only quadratic-type tangencies with $\mathcal{F}$ (i.e. if $D=(h=0)$ then the restriction of $h$ to a leaf has multiplicity less or equal to $2$ at any point) then $\mathcal{G}$ has canonical singularities.
\end{enumerate}
\end{proposition}

\begin{proof}
\begin{enumerate}
\item If $\mathcal{F}$ is transverse to $D$, we can choose local coordinates $(z_1,\dots,z_n)$ on $\mathcal{U}$ such that $\mathcal{F}$ is given by $dz_1=0$  and $D=(z_2=0).$ The covering $\pi$ is given by $(z_1,t,\dots,z_n) \to (z_1,t^m,\dots,z_n)$ and therefore $\mathcal{G}$ is given by $dz_1=0.$

\item If $D$ has only isolated quadratic-type tangencies with $\mathcal{F}$, locally by the holomorphic Morse lemma, we can find local coordinates such that $h$ the local equation defining $D$ is written,
$$h(z)=z_1+\sum_{i=2}^nz_i^2.$$
So, in local coordinates $(t,z_2,\dots,z_n)$, $\mathcal{G}$ is given by $d(t^m-\sum_{i \geq 2} z_i^2)=0.$ 
Now, from \cite{Reid}, we know that hypersurfaces given by $t^m-\sum_{i \geq 2} z_i^2=0$ have at most canonical singularities if $\frac{1}{m}+\frac{n-1}{2}>1$. Then we conclude by proposition \ref{lc}

\item If $S$ is a germ of leaf, then by the hypotheses the pair $(S,\left(1-\frac{1}{2}\right)D|_S)$ is canonical and we conclude by Lemma \ref{ramcan}.
\end{enumerate}
\end{proof}

Now, let us see how we can ensure globally the conditions of the previous proposition for a smooth divisor $D$ in a complex projective manifold $X$.

Let $X_1:=G(n-1,T_X)$ be the Grassmannian bundle. Then the inclusion $i : D \hookrightarrow X$ induces a lift $i_1: D \to X_1$.
A smooth codimension-one foliation $\mathcal{F}$ on $X$ gives a section $F \subset X_1$.

The condition for $\mathcal{F}$ to be transverse to $D$ is equivalent to $$i_1(D)\cap F = \emptyset.$$

The condition for $D$ to have only isolated non-degenerate quadratic-type tangencies with $\mathcal{F}$ translates into the fact that $i_1(D)\cap F$ is finite and $i_1$ has some non-degeneracy property over that set.

Let us illustrate these situations in the case where $X$ is an abelian variety.

\begin{prop}\label{prop:pullback}
Let $A$ be an $n$-dimensional complex abelian variety and $L$ a line bundle on $A$. Let $D$ be a smooth divisor in the linear system $|mL|$, where $m>1$. Let $\pi:X\to A$ be the degree $m$ cyclic cover of $A$ ramified along $D$. If $\cF$ is a generic linear codimension one foliation on $A$, the induced foliation $\cG:=\pi^*\cF$ on $X$ has at most isolated canonical singularities.
\end{prop}
\begin{proof}
By the triviality of the tangent bundle $T_A$, the lifting $i_1: D \to X_1$ described above yields
a map
$$
\phi : D\to \bP^{n-1},
$$
which is locally defined by $[\partial h/\partial z_1:\ldots:\partial h/\partial z_n]$, where $\{h=0\}$ is a local equation for $D$.
If we choose a base $dz_1,\ldots,dz_n$ of invariant differentials, the foliation  $\cF$ corresponds to a point $p\in\bP^{n-1}$.
We have two cases.
If $\phi$ is not dominant, then for a generic choice of $\cF$ the preimage $\phi^{-1}(p)$ will be empty, that is, $D$ is transverse to $\cF$. Therefore the pullback $\pi^*\cF$ is smooth, by the local result above, item (1).

If $\phi$ is dominant, for a  generic choice of $\cF$,
we have that $\phi^{-1}(p)$ is smooth. Notice that we may assume that $D$ is given locally by the equation $\{z_1 + Q(z_1,\ldots,z_n)=0\}$, where $\deg(Q)\geq 2$ (if not, $D$ is transverse to $\cF$ and we are done as before). Then, the smoothness of $\phi^{-1}(p)$ is equivalent to the fact that $(\partial^2Q/\partial z_i\partial z_j)$ is non-degenerate. Therefore we are done by the local result above, item (2).
\end{proof}
\begin{example}
Let $L$ be an ample line bundle  on an abelian variety $A$. Let $D\in|mL|$ be a smooth divisor and
$\pi:X\to A$ be the degree $m$ cyclic cover ramified along $D$. Observe that $X$ is of general type. Take a generic linear codimension one foliation $\cF$ on $A$.
If $n:=\dim(A)\geq 3$, by Proposition \ref{prop:pullback} the foliated pair $(X,\pi^*\cF)$ satisfies the hypotheses of Corollary D.  In particular $X$ contains no maximal rank holomorphic mapping $f:\bC^{n-1}\to X$ tangent to $\pi^*\cF$.
\end{example}

\begin{remark}
Using \cite[Main Theorem]{NWY} one may obtain that $f:\bC^{n-1}\to X$ is not Zariski dense in the previous example.
\end{remark}

Now we turn back to the condition $(3)$ in the local result above i.e. when $D$ has only quadratic-type tangencies with $\mathcal{F}$. This condition can be verified considering the second jet-bundle $X_2$ as defined in \cite{PaRou} which we refer to for details. $X_2$ is a Grassmannian bundle over $X_1$. We have a lifting $i_2: D \to X_2$ and the foliation defines a section $s_2: X \to X_2$. The previous condition translates as
$$i_2(D)\cap s_2(X) = \emptyset.$$

\section{The dimension $3$ case}

\subsection{Desingularization of currents}
In the case of surfaces, the following property of the current is quite useful and often used in \cite{McQ0} and \cite{Br}. Consider a tower of blow-up mappings $\pi_k: X_k \to X_{k-1}$, with the convention that $X_0=X$ the original surface and $E_k$ denotes the corresponding exceptional divisor. Then
$$\lim_{k \to \infty} [\Phi_k].[E_k]=0,$$
where $[\Phi_k]$ is the current associated to the lifting of $f: \bC \to X$.

In fact, McQuillan \cite{Bloch} proves the following stronger result.
\begin{theorem}\label{blo}
Let $(\mathcal{Y},\mathcal{F})$ be a foliated non-singular surface with canonical singularities and $b$ a formal branch of $\mathcal{F}$ through a canonical singularity $z$. Let $\mathcal{Y}_1$ be obtained by blowing up $\mathcal{Y}_0=\mathcal{Y}$ in $z$, with $b_1$ the proper transform of the branch $b_0=b$, $\mathcal{Y}_2$ from $\mathcal{Y}_1$  by blowing up the crossing point of $b_1$ with the exceptional divisor etc., and $\mathcal{E}_k$ the exceptional curve on $\mathcal{Y}_k$ blown down by $\rho_{k,k-1}$ ($\rho_{km}$ from $\mathcal{Y}_k$ to $\mathcal{Y}_m$, $m <k$ being the projection) then for $H$ ample there is a positive rational constant $\alpha$ such that for all $k$, every sufficiently large and divisible multiple of $\alpha (\sqrt(k)^{-1}\rho_{k0}^*H-\mathcal{E}_k$ is generated by its global sections outside the exceptional curves other than $\mathcal{E}_k$.
\end{theorem}

In this section we would like to prove an analoguous desingularization statement for the currents constructed above, in the case $\dim X= 3$, that will be used in a sequel of this paper.

We have to study the following situation. Let $(X, \mathcal{ F})$ be a smooth projective variety of dimension three equipped with a foliation of codimension one. Let $Z$ be a smooth leaf of the foliation and $C$ be a compact curve contained in it. 

We perform the following construction. We denote $X=X_0$, $V=V_0$ and $C=C_0$. Let $f_{1}:X_1\to X_0$ be the blow up along $C_0$, $E_1$ be the exceptional divisor and $Z_1\subset X_1$ the strict transform of $Z_0$. Let $C_1:=E_1\cdot Z_1$.  Inductively we define the sequence $(X_k,  C_k, E_k, Z_k)$  where: $f_{k,k-1}:X_k\to X_{k-1}$ is the blow up of $X_{k-1}$ along $C_{k-1}$, $E_k$ is the exceptional divisor  of $f_{k,k-1}$ and $Z_k$ is the strict trasform of $Z_{k-1}$. We will denote by $f_k$ the natural projection $f_k:X_k\to X_0$ and, for every $i<k$, we denote by $f_{k,i}$ the natural projection $f_{k,i}:X_k\to X_i$. Let $i<k$, by abuse of notation, we will denote by $E_i$ the divisor $f_{k,i}^\ast(E_i)$. We fix an ample divisor $H$ on $X_0$ and we denote by $H$ the nef divisor $f_k^\ast(H)$ on $X_k$.

Then we want to prove the following.
\begin{theorem}\label{effectiveandblowups} In the situation above, there is a constant $\alpha$ independent of $k$ such that, for every $k$ and $m$ sufficiently big (possibly depending in $k$), the divisor on $X_k$ 
$$m\left({{\alpha}\over{k^{1/3}}}H-E_k\right)$$
is effective.
\end{theorem}

For this we need the following.

\begin{proposition}\label{effective1} With the notation above, for $\alpha$ sufficiently big, the line bundle $$F_{\alpha,k}:=\alpha k^{2/3}H-\sum_{i=1}^kE_i$$ is effective on $X_k$.
\end{proposition}
The proof consists in four steps which will be stated in the lemmas below. But first let us show that this proposition implies the theorem.

\begin{proof}[Proof of Theorem \ref{effectiveandblowups}] Using the notations below, denoting by $\overline{E_i}$ the strict transform of $E_i$ in $X_k$ and remarking that $\sum_{i=1}^kE_i=\sum_{i=1}^ki\overline{E}_i$, we get
$${{\alpha}\over{k^{1/3}}}H-E_k={{1}\over{k}}\left(\left(\alpha k^{2/3}H-\sum_{i=1}^kE_i\right)+\sum_{j=1}^{k-1}j\overline{E}_j\right).$$
The conclusion easily follows because the two terms on the right are effective divisors.
\end{proof}

The {\it first step} is

\begin{lemma} 
If $\alpha$ is sufficiently big independently of $k$, then $F_{\alpha,k}^3>0$. 
\end{lemma}

\begin{proof}
A systematic use of the projection formula gives the following relations:
$$(H^2,E_i)=0,$$ $$(H,E_i^2)=(H,C_0)=r_0 \text{  for a suitable } r_0>0,$$ $$(E_i,E_j,E_k)=0 \text{ if } k>j>i;$$ $$(E_i^2,E_j)=0  \text{ if } j>i,$$ and $$(E_i^2,E_j)=-(E_j,C_j) \text{ if } j<i.$$ Since $f_{i,i-1}^\ast(Z_{i-1})=Z_i+E_i$, again the projection formula gives $$(E_i,C_i)=-E_i^3.$$ The relations above gives the existence of a positive constant $s_0$ independent on $k$ such that
\begin{eqnarray*}F_{\alpha, k}^3 &= \alpha^3k^2H^3-\alpha k^{2/3}r_0k-\sum_{i=1}^kE_i^3-3\sum_{i=1}^k\sum_{j=i+1}^k(E_i,E_j^2)\\
&=\alpha k^2s_0-\sum_{i=1}^kE_i^3(1+3(k-(i+1)).\\
\end{eqnarray*}

Consequently $F_{\alpha,k}^3>0$ for a sufficiently big $\alpha$, as soon as we prove that there is a constant $R$ independent of $k$ such that, for every $i$ we have that $\vert E_i^3\vert\leq R$.

Let $N_i$ be the restriction of the normal sheaf of $Z_i$ to $C_i$. Since $C_i$ is the intersection of $E_i$ and $Z_i$, we obtain that 
$$-E_i^3=\deg(N_{i-1}(E_{i-1})\vert_{C_{i-1}}).$$
Again, by projection formula we get
$$-E_i^3=(N_{i-2},C_{i-2}).$$
Thus, using the fact that $-E_j^3=(E_j,C_j)$, by induction we get
\begin{eqnarray*}-E_i^3-E_{i-2}^3-E_{i-3}^3-\dots-E_1^3=-a,\\
-E_{i-1}^3-E_{i-3}^3-\dots-E_1^3=-a,\\
\dots\\
-E_3^3-E_1^3=-a,\\
-E^3_2=-a,\\
-E_1^3=-b.\\
\end{eqnarray*}

From this we get the recursive formula 
\begin{eqnarray*} E_i^3=E_{i-1}^3-E_{i-2}^3,\\
-E^3_2=-a,\\
-E_1^3=-b.\\
\end{eqnarray*}
Consequently the $E_i^3$ are periodic thus $F_{\alpha,k}^3>0$ for $\alpha\gg 0$ independent of 
$k$.
\end{proof}


Now, we state the {\it second step} as the following lemma which is interesting in its own.

\begin{lemma}\label{nefonrestriction} Let $X$ be a smooth projective variety of dimension $n$. Let $L$ be a line bundle  and $D$ a smooth nef and big divisor on $X$. Suppose that the restriction $L\vert_D$ of $L$ to $D$ is nef and big.  Then, there exist $a_0$ such that for every $m\gg 0$ and $i>1$ we have $h^i(X,L^{ma_0})=0$.
\end{lemma}
\begin{proof}  Let $K_X$ be the canonical line bundle of $X$. By Serre duality it suffices to show that $h^j(X,L^{-{ma_0}}+K_X)=0$ for $j\leq n-2$ and $m>0$ sufficiently big.

By hypotheses,  we may suppose that $a_0$ is so big that both line bundles $(L^{ma_0} - K_X)\vert_D$ are nef and big on $D$. Thus, by Kawamata--Viehweg vanishing theorem,  $h^j(D,L^{-ma_0} + K_X)=0$, as soon as $m>0$. In the sequel we will denote by $L$ the line bundle $L^{a_0}$. 

Consider the exact sequence
$$0\longrightarrow (L^{-m}+K_X-(i-1)D)\vert_D\longrightarrow L^{-m}+K_X\vert_{iD}\longrightarrow  L^{-m}+K_X\vert_{(i-1)D}\longrightarrow 0.$$
Since, for every $i>0$ and $m>1$, the line bundle $(L^m-K_X+iD)\vert_D$ is nef and big, again by Kawamata--Viehweg vanishing theorem, for every $j<n-1$ we have $h^j(D, L^{-m}+K_X-iD)=0$. Consequently, by induction on $i$ we have $h^j(iD, L^{-m}+K_X)=0$.

Fix $m>0$. We can take $i$ so big that $L^m-K_X+iD$ is nef and big on $X$ ($i$ will depend on $m$ in general). 

Kawamata--Viehweg vanishing theorem applied to $L^m-K_X+iD$ and the long exact sequence of cohomology of the exact sequence
$$0\longrightarrow L^{-m}+K_X-iD\longrightarrow L^{-m}+K_X\longrightarrow (L^{-m}+K_X)\vert_{mD}\longrightarrow 0$$ 
allow to conclude.
\end{proof}

The {\it third step} is
\begin{lemma}
Suppose that $D$ is a generic  global section of (a power of) $H$.  Let $D_k$ be the strict transform of $D$ in $X_k$. Then, the restriction of $F_{\alpha,k}$ to $D_k$ is nef and big. 
\end{lemma}

\begin{proof}
Let $\mathcal{F}_D$ be the restriction of the foliation $\mathcal{F}$ to $D$. Let $p_1,\dots p_r$ be the intersection points of $D$ with $C$. They are singular points for the foliation $\mathcal{F}_D$. Let $Z_D$ be the intersection of $Z$ with $D$. It is a leaf of the foliation $\mathcal{F}_D$. We may perform a tower of blow ups $(D_k, (E_D)_k, p_k, (Z_D)_k)$  of $D$ on the points $p_j$ similar to the tower we performed on $X$. 
Since $D$ is transverse to the curve $C$, the strict transform of $D$ in $X_k$ is exactly $D_k$ and the intersection of $E_k$ with $D_k$ is $(E_D)_k$. Since Seidenberg's desingularization process is obtained by blow up, as soon as $k$ is sufficently big, we may suppose that the restriction of the foliation to $D_k$ is with reduced singularities on the points $p_k$. 

The restriction of $F_{\alpha, k}$ to $D_k$ is then nef and big as soon as $\alpha$ is sufficiently big by Theorem \ref{blo}.
\end{proof}

Finally, the {\it fourth step} is the proof of  Proposition \ref{effective1}.

\begin{proof}[Proposition \ref{effective1}]
The Euler characteristic $\chi(F_{\alpha,k})$ is positive because $F_{\alpha,k}$ has positive self intersection. Since, the restriction of $F_{\alpha,k}$ to the nef and big smooth divisor $D_k$ is nef and big, then $h^i(X_k, F_{\alpha,k}^m)=0$ for $m\gg 0$ and $i>1$. Thus, as soon as $k$  and $m$ are sufficiently big, we have $h^0(X_k, F_{\alpha,k}^m)>0$.
\end{proof}

\begin{remark} (a) We proved a little bit more: for  $\alpha$ and $k$ sufficiently big, the line bundle ${{\alpha}\over{k^{1/3}}}H-E_k$ is a big divisor.

(b) If one wants to avoid the Seidenberg theorem on resolution of singularities of foliations on surfaces, one can remark that in the case of logarithmic simple singularities, we may suppose that $D$ intersect the curve $C$ properly and only on smooth points, we may suppose that the foliation $\mathcal{F}_D$ has reduced singularities on $D$. 

\end{remark}

As a corollary of theorem \ref{effectiveandblowups} we obtain the following desingularization statement.
\begin{corollary}\label{limite}
In the situation above, consider a Zariski-dense holomorphic map $f:\bC^2 \to X$ and its (meromorphic) liftings $f_k:\bC^2 \to X_k$ with the associated currents $[\Phi]^{(k)}$. Then
$$\lim_{k \to + \infty} [\Phi]^{(k)}.E_k= 0.$$
\end{corollary}

\begin{proof}
From  Theorem \ref{effectiveandblowups}, we have
$$ 0 \leq [\Phi]^{(k)}.E_k \leq \frac{\alpha}{k^{1/3}}  [\Phi].H,$$
and we let $k$ tend to infinity.
\end{proof}

\subsection{Degeneracy for canonical foliations}
Let $(X,\mathcal{F}, E)$ be a foliated threefold with canonical singularities adapted to a normal crossing divisor $E$. Thanks to \cite{Can} (see also \cite{CeMo}) we have a list of the local formal models of these singularities:
\begin{itemize}
\item the logarithmic case: the model is
$$\omega=\left(\prod_{i=1}^r z_i \right) \sum_{i=1}^r \lambda_i\frac{dz_i}{z_i};$$
\item the resonant case: the model is
$$\omega=\left(\prod_{i=1}^r z_i^{p_i+1} \right) \left(\sum_{i=1}^r \lambda_i\frac{dz_i}{z_i}+d\left(\frac{1}{z_1^{p_1}\dots z_r^{p_r}}\right) \right);$$
\end{itemize} 

where $r$ is the dimensional type of the foliation. 

In this setting, we want to prove the following 
\begin{theorem}
$$[\Phi].K_\mathcal{F}^{-1} \geq 0.$$
\end{theorem}

\begin{proof}
We consider $X_1:= \bP(\overset{2} \bigwedge  T_X(-\log(E))) \cong \bP(T_X^*(\log(E)))$ with its projection $\pi: X_1 \to X$ and look at the graph of the foliation $\tilde{X} \subset X_1$. Let us look at the singularities that may appear above $Sing \mathcal{F}$. We concentrate on the $3$-dimensional type, since for the $2$-dimensional type we recover the same properties as in dimension $2$ studied in \cite{McQ0}.
From the list above, locally at $3$-dimensional type singularities, we have

\begin{itemize}
\item logarithmic simple singularities given by:
$$z_1z_2z_3\left(\sum_{i=1}^3 \lambda_i \frac{dz_i}{z_i}\right),$$
where $\frac{\lambda_i}{\lambda_j} \not \in \bQ_{<0}.$

\item resonant simple singularities given by:
$$z_1z_2z_3 \left(z_1^{p_1}z_2^{p_2}z_3^{p_3}\sum_{i=1}^3 \lambda_i \frac{dz_i}{z_i}-\sum_{i=1}^3 p_i \frac{dz_i}{z_i}\right),$$
where $p_1p_2p_3 \neq 0.$

\item resonant saddle-node simple singularities given by
$$z_1z_2z_3 \left(z_1^{p_1}z_2^{p_2}\sum_{i=1}^3 \lambda_i \frac{dz_i}{z_i}-\sum_{i=1}^2 p_i \frac{dz_i}{z_i}\right),$$
where $\lambda_3p_1p_2 \neq 0$.

\item logarithmic saddle-node simple singularities given by
$$z_1z_2z_3 \left(z_1^{p_1}\sum_{i=1}^3 \lambda_i \frac{dz_i}{z_i}- p_1 \frac{dz_1}{z_1}\right),$$
where $p_1 \neq 0$, $\lambda_2\lambda_3\neq 0.$
\end{itemize}
As the foliation is adapted to $E$ we can suppose that $z_1,z_2$ correspond to algebraic components of $E$. 
We see that in the case of logarithmic simple singularities and resonant simple singularities $\tilde{X}$ is smooth above $Sing \mathcal{F}$.

In the case of logarithmic saddle-node simple singularities, we see that $\tilde{X}$ is singular above the $z_2$- axis: it is the blow-up in the non-reduced ideal $(z_1^{p_1}, z_3)$. So if we blow-up $X$ successively $p_1$ times along the curves corresponding to the strict transforms of the $z_2$- axis, we get a resolution $X^1 \to \tilde{X}$ of $\tilde{X}$.

In the case of resonant saddle-node simple singularities, we see that $\tilde{X}$ is singular above the curve which is the union of the $z_2$- axis and the $z_1$- axis: it is the blow-up in the non-reduced ideal $(z_1^{p_1}z_2^{p_2}, z_3)$. So if we blow-up $X$ successively $p_1$ times along the curves corresponding to the strict transforms of the $z_2$- axis and $p_2$ times along the curves corresponding to the strict transforms of the $z_1$- axis, we get a resolution $X^1 \to \tilde{X}$ of $\tilde{X}$.

One notes that the resolution just described depends only on the formal model of the foliation.

We have
$$\mathcal{O}_{X_1}(-1)_{|\tilde{X}}=\pi^*K_{\mathcal{F}}^{-1}\otimes \mathcal{O}(E_0),$$
where $E_0$ is the total exceptional divisor, above the intersection of the analytic leaf with $E$, on $\tilde{X}$ seen as a blow up along (possibly non-reduced) curves as we have seen.

So, the logarithmic tautological inequality and lemma \ref{Edoesntcount} imply
$$[\Phi].K_\mathcal{F}^{-1}\geq -[\Phi_1].E_0.$$

The discussion above gives a procedure to obtain a resolution $X^1$ of $\tilde{X}$ by blowing up $X$ successively along curves. Moreover, $[\Phi_1].E_0 \leq \sum [\Phi^k].E^k$ where the $E^k$ are the exceptional divisors coming from the successive blow-ups defining  $X^1$.

On $X^1$ we get a foliation $\mathcal{F}^1$ and again
$$[\Phi]^{(1)}.K_{\mathcal{F}^1}^{-1}\geq -[\Phi_1]^{(1)}.E_1.$$
We obtain by the same argument a resolution $X^2$ of the graph of $\mathcal{F}^1$ and by induction we get $X^n$, $\mathcal{F}^n$ and
$$[\Phi]^{(n)}.K_{\mathcal{F}^n}^{-1}\geq -[\Phi_1]^{(n)}.E_n.$$

But since we blow-up only canonical singularities, we have
$$[\Phi]^{(n)}.K_{\mathcal{F}^n}^{-1}\leq [\Phi].K_\mathcal{F}^{-1}.$$

From the above remark and the fact that at each step the number of blow-ups is the same, we can use corollary \ref{limite} to obtain
$$\lim_{n \to + \infty} [\Phi_1]^{(n)}.E_n= 0,$$
which finishes the proof.

\end{proof}

The last theorem implies Theorem F.

\bigskip

\noindent
{\tt gasbarri@math.unistra.fr}\\
{\tt pacienza@math.unistra.fr}\\
{\tt rousseau@math.unistra.fr}\\
Institut de Recherche Math\'ematique Avanc\'ee\\
Universit\'e de Strasbourg et CNRS\\
7, rue Ren\'e Descartes, 67084 Strasbourg C\'edex - FRANCE

\end{document}